\documentclass[12pt]{amsart}
\usepackage{amssymb,amsmath,amscd,graphicx,
latexsym,amsthm}
\usepackage{amssymb,latexsym,eufrak,amsmath,amscd,graphicx}
  \usepackage[all]{xy}
  \usepackage{pdfsync}
\theoremstyle{plain}
\newtheorem{theorem}{Theorem}
\newtheorem{proposition}[theorem]{Proposition}

\newtheorem{lemma}[theorem]{Lemma}

\newtheorem{proposition.definition}[theorem]{Proposition/Definition}

\newtheorem{theoremalpha}{Theorem}
\newtheorem{corollaryalpha}[theoremalpha]{Corollary}

\theoremstyle{definition}

\newtheorem{remark}[theorem]{Remark}

\newcommand{\lra}{\longrightarrow}

\newcommand{\noi}{\noindent}
\newcommand{\PP}{\mathbf{P}}

\newcommand{\CC}{\mathbf{C}}

\newcommand{\OO}{\mathcal{O}}

\newcommand{\FF}{\mathcal{F}}
\newcommand{\EE}{\mathcal{E}}

\newcommand{\frakm}{\mathfrak{m}}

\newcommand{\eps}{\varepsilon}

\newcommand{\HH}[3]{H^{{#1}} \big( {#2} , {#3}
\big) }

\newcommand{\coker}{\textnormal{coker}}

\newcommand{\pro}{{pr}}

\newcommand{\Sym}{\textnormal{Sym}}

\newcommand{\ev}{\textnormal{ev}}

\newcommand{\ol}[1]{\overline{#1}}
\newcommand{\SSS}{\mathcal{S}}
\newcommand{\NNN}{\mathcal{N}}
\newcommand{\GGG}{\mathcal{G}}
\newcommand{\FFF}{\mathcal{F}}

\numberwithin{equation}{section}
\numberwithin{theorem}{section}

\begin{document}

\title{A vanishing theorem for weight one syzygies}

 \author{Lawrence Ein}
  \address{Department of Mathematics, University of Illinois at Chicago, 851 South Morgan St., Chicago, IL  60607}
 \email{{\tt ein@uic.edu}}
 \thanks{Research of the first author partially supported by NSF grant DMS-1501085.}

 \author{Robert Lazarsfeld}
  \address{Department of Mathematics, Stony Brook University, Stony Brook, New York 11794}
 \email{{\tt robert.lazarsfeld@stonybrook.edu}}
 \thanks{Research of the second author partially supported by NSF grant DMS-1439285.}

 \author{David Yang}
 \address{Department of Mathematics, MIT, Cambridge, MA 02139}
 \email{{\tt dhy@mit.edu}}
 
\maketitle
  
  \begin{abstract}
We give a criterion for the vanishing of the weight one syzygies associated to a line bundle $B$ in a sufficiently positive embedding of a smooth complex projective variety of arbitrary dimension.    \end{abstract}
    
 \section*{Introduction}

Inspired by the methods of Voisin in \cite{Voisin1} and \cite{Voisin2},  the first two authors recently proved  the gonality conjecture of \cite{GL}, asserting that one can read off the gonality of an algebraic curve $C$ from the syzgies of its ideal in any one embedding of sufficiently large degree. This was deduced  in \cite{EL2} as a special case of a vanishing theorem for the asymptotic syzygies associated to an arbitrary line bundle $B$ on $C$, and it was conjectured there that an analogous statement should hold on a smooth projective variety of any dimension. The purpose of this  note is to prove the conjecture in question.
 
 Turning to details, let $X$ be a smooth complex projective variety of dimension $n$, and set 
 \[ L_d \ = \ dA + P, \]
 where $A$ is ample and $P$ is arbitrary. We always assume that $d$ is sufficiently large so that $L_d$ is very ample, defining an embedding
 \[  X \ \subseteq \ \PP H^0(X, L_d) \, = \, \PP^{r_d}. \] Given an arbitrary line bundle $B$ on $X$, we wish to study the weight one syzygies of $B$ with respect to $L_d$ for $d \gg 0$. More precisely, let $S = \Sym \, H^0(X, L_d)$ be the homogeneous coordinate ring of $\PP H^0(X, L_d)$, and put
 \[   R \ = \  R(X, B; L_d) \ = \ \bigoplus_m \, H^0(X, B + mL_d). \]
 Thus $R$ is a finitely generated graded $S$-module, and hence has a minimal graded free resolution $E_\bullet = E_\bullet(B; L_d)$:
  \[
 \xymatrix{
0 \ar[r] & E_{r_d} \ar[r]& \ldots \ar[r] &    E_2 \ar[r] & E_1 \ar[r]  & E_0 \ar[r] & R \ar[r] & 0 , }\]
where $ E_p =\oplus S(-a_{p,j})$.  
 As customary, denote by $K_{p,q}(X, B;L_d)$ the finite-dimensional vector space of degree $p+q$ minimal generators of $E_p$, so that 
 \[  E_p(B; L_d)  \ = \ \bigoplus_q K_{p,q}(X, B; L_d) \otimes_\CC S(-p-q).  \]  We refer to elements of this group as $p^\text{th}$ syzygies of $B$ with respect to $L_d$ of  weight $q$. When $B = \OO_X$ we write simply $K_{p,q}(X; L_d)$, which -- provided that $d$ is large enough so that $L_d$ is normally generated --  are the vector spaces describing the syzygies of the homogeneous ideal $I_X \subset S$ of $X$ in $\PP H^0(X, L_d)$.  
 
 The question we address involves fixing $B$ and asking when it happens that 
 \[  K_{p,q}(X, B; L_d) \, = \, 0 \ \ \text{for } \, d \gg 0. \]
 When $q = 0$ or $q \ge 2$ the situation is largely understood thanks to results of \cite{Kosz1}, \cite{Kosz2}, \cite{SAD} and \cite{ASAV}. (See Remark \ref{Other.Kpq.Remark} for a summary.) 
Moreover in this range the statements are uniform in nature, in that they don't depend on the geometry of $X$ or $B$. However as suggested in \cite[Problem 7.1]{ASAV}, for $K_{p,1}(X, B; L_d)$ one can anticipate more precise asymptotic results that do involve geometry. This is what we establish here.

Recall that a line bundle $B$ on a smooth projective variety $X$ is said to be \textit{$p$-jet  very ample} if for every effective zero-cycle
\[   w \ = \ a_1 x_1 \, + \, \ldots \, + \, a_s x_s \]
of degree $p+1 = \sum a_i$ on $X$, the natural map
\[   \HH{0}{X}{B}  \lra \HH{0}{X}{B \otimes  \OO_X/  \frakm_1^{a_1} \cdot \ldots \cdot \frakm_s^{a_s} } \]
is surjective, where $\frakm_i  \subseteq \OO_X$ is the ideal sheaf of $x_i$. So for example  if $p=1$ this is simply asking that $B$ be very ample. When  $\dim X = 1$ the condition is the same as requiring that $B$ be $p$-very ample -- i.e. that every subscheme $\xi \subset X$ of length $p+1$ imposes independent conditions in $H^0(X, B)$ -- but in higher dimensions it is a stronger condition. 

Our main result is
\begin{theoremalpha} \label{Main.Thm.A}
If $B$ is $p$-jet  very ample, then 
\[  K_{p,1}(X, B; L_d) \ = \ 0 \ \text{ for } \ d \gg 0. \]
\end{theoremalpha}
\noi The statement was conjectured in \cite[Conjecture 2.4]{EL2}, where the case $\dim X = 1$ was established. 

It is not clear whether one should expect that  $p$-jet amplitude is  {equivalent} to the vanishing of $K_{p,1}(X, B; L_d)$ for $d \gg 0$. However we prove:
\begin{theoremalpha} \label{Non.Van.Prop.Intro}
Suppose that there is a \textnormal{reduced} $(p+1)$-cycle $w$ on $X$ that fails to impose independent conditions on $H^0(X,B)$. Then
\[ K_{p,1}(X, B; L_d) \ \ne \ 0 \ \ \text{for \ } d \gg 0. \]
\end{theoremalpha}  
\noi In general, the proof of Theorem \ref{Main.Thm.A} will show that if $H^1(X,B) = 0$, then the jet amplitude hypothesis on $B$ is equivalent when $d \gg 0$ to the vanishing of a group that contains $K_{p,1}(X,B; L_d)$ as a subspace (Remark \ref{Tensor.Vanishing.Equivalence}). 

When $B = K_X$ is the canonical bundle of $X$, Theorem A translates under a mild additional hypothesis into a statement involving the syzygies of $L_d$ itself. 
\begin{corollaryalpha} \label{Corollary.Main.Result}
Assume that $H^i(X, \OO_X) = 0$ for $0 < i < n$, or equivalently that $X \subseteq \PP^{r_d}$ is projectively Cohen-Macaulay for $d \gg 0$. 
\begin{enumerate}
\item[(i).] The canonical bundle $K_X$ of $X$ is very ample if and only if
\[  K_{r_d-n-1,n}(X;L_d) \ = \ 0 \ \text{ for } \ d \gg 0. \]
\vskip 5 pt
\item[(ii).] If $K_X$ is $p$-jet very ample, then
\[ K_{r_d -n -p,n}(X;L_d) \ = \ 0  \ \text{ for } \ d \gg 0. \]
\end{enumerate}
\end{corollaryalpha}
\noi When $n = \dim X =1$, this (together with Theorem \ref{Non.Van.Prop.Intro}) implies that $K_{r_d - c,1}(X; L_d) \ne 0$ for $d \gg 0$ if and only if $X$ admits a branched covering $X \lra \PP^1$ of degree $\le c$, which is the statement of the gonality conjecture established in \cite{EL2}.

The proof of Theorem \ref{Main.Thm.A}  occupies \S 1. It follows very closely the strategy of \cite{EL2}, which in turn was inspired by the ideas of Voisin in    \cite{Voisin1} and \cite{Voisin2}. However instead of working on a Hilbert scheme or symmetric product, we work on a Cartesian product of $X$, using an idea that goes back in a general way to Green \cite{Kosz2}. For the benefit of non-experts, we outline now the approach in some detail  in the toy case $p = 0$.\footnote{This was in fact the train of thought that led us to the arguments here and in \cite{EL2}.}

Keeping notation as above, it follows from the definition that $K_{0,1}(X, B; L_d) = 0 $ if and only if the multiplication map
\[  \HH{0}{X}{B} \otimes \HH{0}{X}{L_d} \lra \HH{0}{X}{B \otimes L_d} \tag{*}\]
is surjective: in fact, $K_{0,1}$ is its cokernel. A classical way to study such maps is to pass to the product $X \times X$ and then restrict to the diagonal. Specifically, (*) is identified with the homomorphism
\[ \HH{0}{X\times X}{\pro_1^*B \otimes \pro_2^* L_d } \lra \HH{0}{X\times X}{\pro_1^*B \otimes \pro_2^* L_d \otimes \OO_\Delta }   \tag{**} \]
 arising from this restriction. Thus the vanishing $K_{0,1}(X, B; L_d) = 0$ is implied by the surjectivity of (**). Green observed in \cite{Kosz2} that there is a similar way to tackle the $K_{p,1}$ for $p \ge 1$: one works on the $(p+2)$-fold product $X^{p+2} = X \times X^{p+1}$ and restricts to a suitable union of pairwise diagonals (Proposition \ref{Lemma.Surjectivity.Sufficient}). This is explained in \S1, and forms the starting point of our argument. Although   not strictly necessary we give a new proof of Green's result here  that clarifies its relation to other approaches.

There remains the issue of actually proving the surjectivity of (**) for $d \gg 0$ provided that $B$ is $0$-jet very ample, i.e.~  globally generated. For this one starts with the restriction
\[   \pro_1^* B \lra \pro_1^* B \otimes \OO_\Delta \]
of sheaves on $X \times X$ and pushes down to $X$ via $\pro_2$. There results a map of vector bundles
\[  \ev_B:  H^0(X, B) \otimes_\CC \OO_X \lra B \]
on $X$ which is given by evaluation of sections of $B$. Note that $\ev_B$ is surjective as a map of bundles if and only if $B$ is globally generated. The surjectivity in (*) or (**) is then equivalent to the surjectivity on global sections of the map 
\[
H^0(X,B) \otimes L_d \lra B \otimes L_d 
\]
obtained from twisting $\ev_B$ by $L_d$.

Suppose now that $B$ is $0$-jet very ample. The setting $M_B = \ker (\ev_B)$, we get an exact sequence
\[  0 \lra M_B \lra H^0(X, B) \otimes_\CC \OO_X \lra B \lra 0 \]
of sheaves on $X$. Serre vanishing implies that 
\[   \HH{1}{X}{M_B \otimes L_d} \ = \ 0 \]
for $d \gg 0$, and by what we have just said this means that $K_{0,1}(X, B; L_d) = 0$. The proof of Theorem \ref{Main.Thm.A} in general proceeds along analogous lines. We construct a torsion-free sheaf $\mathcal{E}_B = \mathcal{E}_{p+1, B}$ of rank $p+1$ on $X^{p+1}$ whose fibre at $(x_1, \ldots, x_{p+1})$ is identified with
\[   \HH{0}{X}{B \otimes \OO_X/\frakm_1 \cdot \ldots \cdot \frakm_{p+1}}, \]
where $\frakm_i \subseteq \OO_X$ is the ideal sheaf of $x_i$. This comes with an evaluation map
\[  \ev_{p+1,B} : \HH{0}{X}{B}\otimes_\CC \OO_{X^{p+1}} \lra \mathcal{E}_B \]
which is surjective (as a map of sheaves)  if and only if $B$ is $p$-jet very ample.\footnote{Note that the fibre of $B$ at a point $x \in X$ is identified with $B \otimes \OO_X/ \frakm_x$, i.e.~ $\mathcal{E}_{1, B} = B$.} \ Green's criterion for the vanishing of $K_{p,1}(X, B; L_d)$ turns out to be equivalent  to the surjectivity of  the map on global sections resulting from twisting $\ev_{p+1,B}$ by a suitable ample divisor $\mathcal{N}_d$ on $X^{p+1}$ deduced from $L_d$, and this again follows from Serre vanishing. 

Returning to the case $p = 0$, the argument  just sketched actually proves more. Namely for arbitrary $B$ one has an exact sequence
\[  0 \lra \ker (\ev_B) \lra H^0(X, B) \otimes_\CC \OO_X \lra B \lra \coker (\ev_B) \lra 0, \]
and so Serre vanishing shows conversely that if $B$ is not $0$-jet very ample then 
\[ K_{p,1}(X, B; L_d) \ = \ \HH{0}{X}{\coker (\ev_B) \otimes L_d} \ \ne \ 0\tag{***} \]for $d \gg 0$. Unfortunately this does not generalize when $p \ge 1$ because the computations on $X^{p+1}$ lead to groups that contain $K_{p,1}(X,B;L_d)$ as  summands, but may contain other terms as well. (Said differently, Green's criterion is sufficient but not necessary for the vanishing of $K_{p,1}$.) To prove a non-vanishing statement such as Theorem \ref{Non.Van.Prop.Intro}, one needs a geometric interpretation of $K_{p,1}$ itself. Voisin \cite{Voisin1}, \cite{Voisin2} achieves this by working on a Hilbert scheme -- which has the advantage of being smooth when $\dim X =2$ -- while the third author \cite{Yang} passes in effect to the symmetric product.\footnote{Roughly speaking, one is picking out $K_{p,1}$ inside Green's construction as the space of invariants under a suitable action of the symmetric group.} We follow the latter approach for Theorem \ref{Non.Van.Prop.Intro}: we exhibit a sheaf on $\Sym^{p+1}(X)$ whose twisted global sections compute $K_{p,1}(X,B;L_d)$, and we show that it is non-zero provided that there is a reduced cycle that fails to impose independent conditions on $H^0(X,B)$. Then we can argue much as in the case $p = 0$ just described. This is the content of Section 2.

We are grateful to Claudiu Raicu and Bernd Sturmfels for valuable discussions. We particularly profited from conversations with B. Purnaprajna, who suggested to us that one could work on a Cartesian rather than a symmetric product to establish the vanishing we wanted.

\section{Proof of Theorem \ref{Main.Thm.A}} This section is devoted to the proof of Theorem \ref{Main.Thm.A} from the Introduction.

We start by describing the set-up. As above $X$ is a smooth complex projective variety of dimension $n$, and we consider the $(p+2)$-fold product
\[  Y \ =_{\text{def}} \ X \times X^{p+1} \]
of $X$ with itself. For $0 \le i < j \le p+1$  denote by 
\[   \pi_{i,j} : Y \lra X \times X \]
the projection of $Y$ onto the product of the $i$ and $j$ factors. We write $\Delta_{i,j} \subseteq Y$ for the pull-back of the diagonal $\Delta \subseteq X \times X$ under $\pi_{i,j}$, so that $\Delta_{i,j}$ consists of those points $y = (x_0, x_1, \ldots, x_{p+1}) \in Y$ with $x_i = x_j$.

The basic idea -- which goes back to Green \cite{Kosz2} and has been used repeatedly since (eg \cite{Inamdar}, \cite{BEL}, \cite{LPP}, \cite{HwangTo}, \cite{Yang}) -- is to relate syzygies on $X$ to a suitable union of pairwise diagonals on $Y$. Specifically, let 
\begin{equation} \label{Def.of.Z}  Z \ = \ Z_{p+1} \ =  \Delta_{0,1} \cup \ldots \cup \Delta_{0,p+1} \ \subseteq \ X \times X^{p+1} \end{equation}
be the union of the indicated pairwise diagonals, considered as a reduced subscheme. We denote by
\begin{equation} \label{Z.Projection.Notation}   q : Z \lra X \ \ , \ \ \sigma : Z \lra X^{p+1} \end{equation}
the indicated projections. 

The importance of this construction for us is given by:
\begin{proposition} \label{Lemma.Surjectivity.Sufficient}
Let $L$ be a base-point-free and $B$ an arbitrary line bundle on $X$, and assume $($for simplicity$)$ that $H^1(X, L) = 0$. If the restriction homomorphism 
\begin{equation}  \label{Restriction.Map.to.Z}  \HH{0}{Y}{B \boxtimes L^{\boxtimes p+1}} \lra \HH{0}{Y}{(B \boxtimes L^{\boxtimes p+1}) \mid  Z}  \end{equation}
is surjective, then
\[ K_{p,1}(X, B;L) \ = \ 0. \]
\end{proposition}

The Proposition was essentially established for instance in \cite{Yang}, but it is instructive to give a direct argument. We start with a Lemma that will also be useful later:

\begin{lemma} \label{Ideal.Of.Z.is.Product}
Writing $I_{Z/Y}$ for the ideal sheaf of $Z$ in $Y$, one has
\begin{align*}
I_{Z/Y} \ &= \ I_{\Delta_{0,1} / Y} \cdot I_{\Delta_{0,2} / Y} \cdot \ldots \cdot I_{\Delta_{0,p+1}/Y} \\ &= \ \bigotimes_{j = 1}^{p+1} \, \pi_{0,j}^* I_{\Delta/X\times X}.
\end{align*}
\end{lemma}
\begin{proof} [Sketch of Proof] This is implicit in 
\cite[Theorem 1.3]{LiLi}, but does not appear there explicitly so we very briefly indicate an argument. The statement is \'etale local, so we can assume $X = \mathbf{A}^n$. By looking at a suitable subtraction map, as in \cite[(1-3)]{LPP}, it then suffices to prove the analogous statement for $Y = {X}^{p+1}$ with $Z$ being the union of the ``coordinate planes" \[ L_i \ = \ \{ (x_1, \ldots , x_i , \ldots,x_{p+1} ) \mid x_i = 0\in \mathbf{A}^n \} \ = \  \{ \mathbf{A}^n \} \times \ldots \times \{ 0 \} \times \ldots \times \{ \mathbf{A}^n \}\ \subseteq \ Y  \]  
$(1 \le i \le p+1)$. For this one can proceed by induction on $p$, writing out explicitly the equations defining each $L_i$.    \end{proof}

\begin{remark}
This is the essential place where we use the hypothesis that $X$ is smooth. We do not know whether the statement of the Lemma remains true for singular  $X$. 
\end{remark}

We now turn to the proof of the Propositon:
\begin{proof}[Proof of Proposition \ref{Lemma.Surjectivity.Sufficient}]
To begin with, it is well known (cf \cite{Kosz1}) that  $K_{p,1}(X,B;L)$ is the cohomology of the Koszul-type complex
\[
\Lambda^{p+1} H^0(L) \otimes H^0(B) \lra \Lambda^p H^0(L) \otimes H^0(B+L) \lra \Lambda^{p-1} H^0(L) \otimes H^0(B + 2L).  \] Moreover this cohomology can in turn be interpreted geometrically in terms of the vector bundle $M_L$ on $X$ defined (as in the Introduction) as the kernel of the evaluation map
\[  \ev_L : H^0(L) \otimes_\CC \OO_X \lra L. \]
Specifically,  $M_L$ sits in an exact sequence of vector bundles
\[  0 \lra M_L \lra H^0(L)\otimes \OO_X \lra L \lra 0  \tag{*} \]
on $X$, 
and then   $K_{p,1}(X,B;L) = 0$ if and only if the  sequence
\[  \label{eqn**} 0 \lra \Lambda^{p+1} M_L \otimes B \lra \Lambda^{p+1}H^0(L) \otimes B \lra \Lambda^p M_L \otimes L \otimes B \lra 0 \tag{**}\]
deduced from (*) is exact on global sections. (See for instance \cite[Lemma 1.10]{GL} or \cite{VBT}.)

On the other hand, consider on $X \times X$ the exact sequence
\[ 0 \lra I_\Delta \otimes \pro_2^*L \lra \pro_2^* L \lra L \otimes \OO_\Delta \lra 0.  \]
As in the Introduction, this  pushes down via $\pro_{1}$ to (*). Therefore one finds from Lemma \ref{Ideal.Of.Z.is.Product} and the K\"unneth formula that
\[   q_* \Big(I_{Z/Y} \otimes  B \boxtimes L^{\boxtimes p+1} \Big ) \ = \ \big( \otimes^{p+1} M_L\big) \otimes B, \]
and moreover the $R^1 q_*$ vanishes thanks to our hypothesis that $H^1(L) = 0$.\footnote{The K\"unneth theorem in play here is the following:  let $V_1 \lra S, \ldots,   V_r \lra S$ be mappings of schemes over a field, and suppose that ${F_i}$ is a quasi-coherent sheaf on $V_i$ that is flat over $S$. Write
\[ p_i : V_1 \times_S   \ldots \times_S V_r \lra V_i \  \ , \ \  p: V_1 \times_S   \ldots \times_S V_r \lra S \]
for the natural maps. Then
\[   
p_* \Big( p_1^* F_1 \otimes \ldots \otimes p_r^* F_r \Big) \ = \ p_{1,*} F_1 \otimes \ldots \otimes p_{r,*} F_r \]
as sheaves on $S$, with an analogous K\"unneth-type computations of  the $R^jp_*$. See for instance \cite[Theorem 14]{Kempf} for a simple proof when $S$ is affine, and \cite[Theorem 6.7.8]{EGAIII2} in general.}

\vskip 10pt
Writing
\[ \mathcal  N \ = \ \coker \Big( \otimes^{p+1} M_L \lra \otimes^{p+1} H^0(L) \Big) , \]
it follows that 
\[   q_* \Big ( (B \boxtimes L^{\boxtimes p+1} ) \mid  Z  \Big)  = \mathcal{N} \otimes B, \]
and hence the surjectivity of \eqref{Restriction.Map.to.Z} is equivalent to asking that 
\[  0\lra  \otimes^{p+1} M_L  \otimes B\lra  \otimes^{p+1} H^0(L) \otimes B \lra \mathcal{N} \otimes B \lra 0 \tag{***} \]
be exact on global sections. But since we are in characteristic zero, the exact sequence (**) is a summand of this, and the Lemma follows. \end{proof}

\begin{remark} \label{Tensor.Product.M}
The argument just completed shows that if in addition $H^1(X, B) = 0$ then \eqref{Restriction.Map.to.Z} is surjective if and only if
\[
\HH{1}{X}{ (\otimes^{p+1} M_L) \otimes B } \ = \ 0. 
\]
\end{remark}

It remains to relate these considerations to the jet-amplitude of $B$. To this end, keeping notation as in \eqref{Z.Projection.Notation} set
\[  \EE_B \ = \ \EE_{p+1, B} \ = \ \sigma_* \big( q^* B \big). \]
This is a torsion-free sheaf of rank $p+1$ on $X^{p+1}$ (since it is the push forward of a line bundle under a finite mapping of degree $p+1$), and one has

\begin{lemma} \label{EB.Properties}
\begin{itemize}
\item[(i).] Fix a point \[ \xi \ = \ (x_1, \ldots, x_{p+1} ) \in X^{p+1},\] the $x_i$ being $($possibly non-distinct$)$ points of $X$, and denote by $\EE_B|\xi$ the fibre of $\EE_B$ at $\xi$. Then there is a natural identification
\[  \EE_B|\xi \ = \ \HH{0}{X}{B \otimes \OO_X / \frakm_1 \cdot \ldots \cdot \frakm_{p+1}}, \]
where $\frakm_i \subseteq \OO_X$ is the maximal ideal of $x_i$. 
\vskip 10pt
\item[(ii).] There is a canonical injection
\[   \HH{0}{X}{B} \hookrightarrow \HH{0}{X^{p+1}}{\EE_B}, \]
giving rise to a homomorphism
\begin{equation} \label{evB.Equation} \ev_B = \ev_{p+1,B} : \HH{0}{X}{B} \otimes_{\CC} \OO_{X^{p+1}} \lra \EE_B. \notag \end{equation}
of sheaves on $X^{p+1}$. Under the identification in $($i$)$, $\ev_B$ is given fiberwise by the natural map
\[  \HH{0}{X}{B} \lra \HH{0}{X}{B \otimes \OO_X / (\frakm_1 \cdot \ldots \cdot \frakm_{p+1})}. \]
\vskip10pt
\item[(iii).] The homomorphism \eqref{Restriction.Map.to.Z}  is identified with the map on global sections arising from the sheaf homomorphism \begin{equation} \label{Twisted.ev}
\HH{0}{X}{B} \otimes_{\CC} L^{\boxtimes p+1} \lra \EE_B \otimes L^{\boxtimes p+1}.
\end{equation}
on $X^{p+1}$ determined by twisting $\ev_B$ by $L^{\boxtimes p+1}$.
\vskip 10pt
\item[(iv).] The mapping $\ev_{p+1,B}$ is surjective as a homomorphism of sheaves on $X^{p+1}$ if and only if $B$ is $p$-jet very ample.

\end{itemize}
\end{lemma}
\begin{proof}
For (i), consider the diagram
\begin{equation}   \label{Basic.Diagram}
\xymatrix{
Z   \ar[dr]_\sigma &\subseteq  &    X \times X^{p+1} \ar[dl]^{pr_2}   \\ &  X^{p+1} 
}
 \end{equation}
 and fix $\xi = (x_1, \ldots, x_{p+1} ) \in X^{p+1}$. The scheme-theoretic fibre $\sigma^{-1}(\xi)$ lives naturally as a subscheme of $X$, and  Lemma \ref{Ideal.Of.Z.is.Product} implies that
 it is in fact the scheme defined by the ideal sheaf
 \[ \frakm_1 \cdot \ldots \cdot \frakm_{p+1} \subseteq \OO_X. \]
 Therefore, thanks to the projection formula, the fibre $ \EE_B |\xi$ of $\EE_B$ at $\xi$ is identified with
  \[ \pro_{2,*} \Big( B \otimes \OO_X /(\frakm_1 \cdot \ldots \cdot \frakm_{p+1})\Big) \ = \ \HH{0}{X}{B \otimes \OO_X /(\frakm_1 \cdot \ldots \cdot \frakm_{p+1})}, \]
  as claimed. For (ii), note that in any event
  \[
  \HH{0}{X^{p+1}}{\EE_B} \ = \   \HH{0}{X^{p+1}}{\sigma_* q^* B} \ = \ \HH{0}{Z}{q^* B} .
  \]
On the other hand, each of the irreducible components of $Z$ maps via projection onto $X$, and this gives an inclusion
\[  q^* : \HH{0}{X}{B} \lra \HH{0}{Z}{q^* B} \ = \ \HH{0}{X \times X^{p+1}}{(\pro_1^* B)|Z}.\] 
It is evident from the construction that fiber by fibre $\ev_B$ is as described, and (iv) is then a consequence of the fact that a morphism of sheaves is surjective if and only it is so on each fibre.
Finally, statement (iii)  follows from the construction of $\EE_B$ and $\ev_B$.  \end{proof}

\begin{remark}
Using the resolution of $\OO_Z$ appearing in \cite[p. 4]{Yang}, one can show that in fact
\[  \HH{0}{X^{p+1}}{\EE_B} \ = \ \HH{0}{X}{B}. \]
However this isn't necessary for the argument.
\end{remark}

 \begin{remark}
 The reader familiar with \cite{EL2} or Voisin's Hilbert schematic approach to syzygies will recognize that $Z \lra X^{p+1}$ plays the role of the universal family over the Hilbert scheme, and that Lemma \ref{Lemma.Surjectivity.Sufficient} is the analogue of \cite[Lemma 1, p. 369]{Voisin1}. The sheaf $\EE_{p+1, B}$ plays the role of the vector bundle $E_{p+1, B}$ on the symmetric product appearing in \cite{EL2}.
 \end{remark}
 
 Just as in \cite{EL2}, the main result now follows immediately from Serre vanishing.
\begin{proof}[Proof of Theorem \ref{Main.Thm.A}] 
Assuming that $B$ is $p$-jet very ample, so that $\ev_{p+1, B}$ is surjective, let $\mathcal{M}_{p+1, B}$ denote its kernel:
\[  
0 \lra \mathcal{M}_{p+1, B} \lra H^0(B) \otimes \OO_{X^{p+1}} \lra \EE_B \lra 0. \]
To show that $K_{p,1}(X, B; L_d) = 0$ it suffices thanks to Proposition \ref{Lemma.Surjectivity.Sufficient} and its interpretation in  terms of \eqref{Twisted.ev} to prove that
\[
\HH{1}{X^{p+1}}{\mathcal{M}_{p+1,B}  \otimes L_d^{\boxtimes p+1}} \ = \ 0  
\]
for $d \gg 0$. But this follows immediately from Serre vanishing.\end{proof}

\begin{remark} \label{Tensor.Vanishing.Equivalence}
It follows from the argument just completed that if $L = L_d$ then the surjectivity in Lemma \ref{Lemma.Surjectivity.Sufficient} holds for $d \gg 0$ if and \textit{only if} $B$ is $p$-jet  very ample. In particular, in view of Lemma \ref{Tensor.Product.M} this means that if $H^1(X, B) = 0$ then the $p$-jet very amplitude of $B$ is equivalent to the vanishing
\[ \HH{1}{X}{ (\otimes^{p+1} M_{L_d}) \otimes B } \ = \ 0 \ \text{ for } \ d \gg 0.\]
\end{remark}

 \begin{proof}[Proof of Corollary \ref{Corollary.Main.Result}]
Under the stated hypothesis on $X$, the groups in question are Serre dual to $K_{p,1}(X, B; L_d)$ for $d \gg 0$ (cf \cite[\S 2]{Kosz1}). If $B$ fails to be very ample, then a simple argument as in \cite[Theorem 1.1]{EHU}   shows that $K_{1,1}(X, B; L_d) \ne 0$ for $d \gg 0$, and therefore the Corollary follows from the main theorem. 
\end{proof}

\begin{remark} In the case of curves,  Rathmann \cite{Rathmann} has given a very interesting argument that leads to an essentially optimal effective version of the asymptotic results of \cite{EL2}: in fact, it suffices that $H^1(L) = H^1(L-B) =0$.  In this spirit, it would be very interesting to find an effective estimate for the positivity of $L$  to guarantee the vanishing of $K_{p,1}(X, B;L)$ when $B$ is $p$-jet  very ample. \end{remark}

\begin{remark} \textbf{(Other Koszul cohomology groups).} \label{Other.Kpq.Remark} To conclude this section, we briefly summarize what is known about the groups $K_{p,q}(X, B; L_d)$ for $q \ne 1$. Specifically, fix $B$. Then for $d \gg 0$:
 \begin{enumerate}
\item[(i).] $K_{p,q}(X, B;L_d)  = 0 $ for $q \ge n+2$.
\vskip 5pt
\item[(ii).]  One has 
\begin{align*}
K_{p,0}(X, B; L_d) \ne 0 \ &\Longleftrightarrow \ 0 \le p \le r(B), \text{ and}\\ K_{p,n+1}(X, B;L_d) \ne 0 \ &\Longleftrightarrow \ r_d - n - r(K_X -B) \le p \le r_d -n.
\end{align*}
\vskip 5pt
\item[(iii).] If $q \ge 2$, then $K_{p,q}(X, B; L_d) = 0  \, \text{ when } \, p \le O(d).$
\vskip 5pt
\item[(iv).] For $1 \le q \le n$, $ K_{p,q}(X, B; L_d) \ne 0$ for 
\[   O(d^{q-1}) \  \le \  p \le r_d - O(d^{n-1}). \]
\end{enumerate}
Statement (i) is a consequence of Castelnuovo-Mumford regularity, while (ii) is due to Green and others. (See  \cite[\S3]{Kosz1}, \cite[Corollary 3.3, \S5]{ASAV}.) Assertion (ii) follows for instance from \cite{SAD}, while (iv) is the main result of \cite{ASAV}. 
Furthermore, it is conjectured in \cite{ASAV} that if $q \ge 2$, then $K_{p,q}(X, B; L_d) = 0$ for $p \le O(d^{q-1})$. 
\end{remark}
\section{A non-vanishing theorem}

This section is devoted to the proof of Theorem \ref{Non.Van.Prop.Intro} from the Introduction. Recall the statement:
\begin{theorem} \label{Non-Vanishing.Prop}
Assume that the non-singular projective variety $X$ carries an effective $(p+1)$-cycle $w = x_1 + \ldots + x_{p+1}$ consisting of $p+1$ \textnormal{distinct} points $x_1, \ldots, x_{p+1} \in X$ that fail to impose independent conditions on $H^0(X, B)$. Then
\[  K_{p,1}(X, B; L_d ) \ \ne \ 0 \ \text{ for } \ d \gg 0. \]
\end{theorem}

The argument is somewhat technical, so before launching into  it we would like to outline the rough strategy. As in the case $p = 0$ discussed in the Introduction, in principle we would like to find a map of sheaves on $\Sym^{p+1}(X)$ depending on $B$ --  say
$a_B:  \mathcal{A}_1  \lra \mathcal{A}_2 $
-- having the property that
\begin{equation} \label{Intro2.Eqn} K_{p,1}(X,B; L_d)  \ = \ \coker\Big( H^0(\mathcal{A}_1 \otimes \NNN(L_d)) \lra H^0(\mathcal{A}_2 \otimes \NNN(L_d)) \Big),\end{equation}
where  $\NNN(L_d)$ is a line bundle whose positivity grows suitably with $d$.   Ideally -- as in equation (***) from the Introduction -- we would be able to see that $a_B$ cannot be surjective as a map of sheaves if $B$ is not $p$-jet very ample, and then one could hope to apply Serre vanishing to conclude that $K_{p,1}(X, B;L_d)$ cannot vanish for $d \gg 0$. Unfortunately we do now know whether such a construction is possible.  Instead, what we do in effect  is to use the ideas of the third author from \cite{Yang} to construct a map $\alpha_B$, and show that the non-vanishing of $K_{p,1}$ is implied by the non-vanishing of a certain quotient sheaf of $\mathcal{A}_2$.  We  show that a \textit{reduced} $(p+1)$-cycle that fails to impose independent conditions on $H^0(X,B)$ must appear in the support of this quotient, and this leads to the stated non-vanishing. 

We start by recalling the results \cite{Yang} of the third author interpreting $K_{p,1}$ as an equivariant cohomology group. Consider then a very ample line bundle $L$ on the smooth complex projective variety $X$. Then  the symmetric group $S_{p+1}$ acts in two ways on  the bundle  $L^{\boxtimes p+1}$ on $X^{p+1}$, namely via the symmetric and the alternating characters. Denote these $S_{p+1}$-bundles on $X^{p+1}$ by
\begin{equation} \label{Equivar.Bundles}
L^{\boxtimes p+1, \text{sym}} \ \quad \text{and} \ \quad L^{\boxtimes p+1, \text{alt}} \end{equation}
 respectively. Now let $S_{p+1}$ act on $X\times X^{p+1}$ via the trivial action on the first factor, so that the union of pairwise diagonals $Z \subseteq X \times X^{p+1}$ defined in \eqref{Def.of.Z} becomes an $S_{p+1}$-subspace. It is established in  \cite[Theorem 3]{Yang} that if 
\begin{equation} \label{vanishing.hypothesis}
H^i(X, mL) = H^i(X, B + mL) \ = \ 0 \ \text{for} \ i \, , m \, > 0, \end{equation} then $K_{p,1}(X, B;L)$ is identified with the cokernel of the restriction mapping
\begin{equation} \label{equivariant.restriction.equation}
  H^0_{S_{p+1}}\big( X \times X^{p+1}, \pro_1^* B \otimes \pro_2^* L^{\boxtimes p+1, \text{alt}} \big) \lra H^0_{S_{p+1}}\big( Z, \pro_1^* B \otimes \pro_2^* L^{\boxtimes p+1, \text{alt} } \otimes \OO_Z \big)  \end{equation}
on $S_{p+1}$-equivariant cohomology groups.\footnote{For the theory of equivariant cohomology groups and pushforwards, see \cite[Chapter 5] {Tohuku}. What we need can be summarized as follows. Let $G$ be a finite group acting on a complex projective variety $V$, and suppose given a coherent sheaf $F$ on $X$ together with an action of $G$ on $F$. Then one can define equivariant cohomology groups $H^j_G(V, F)$. While this isn't how they are initially constructed, one can show that 
\[ 
H^j_G(V,F) \ = \ H^j(V,F)^G, 
\] 
the group on the right being the $G$-invariant subspace of $H^j(V,F)$ under the natural action of $G$ on  this cohomology group \cite[p.~ 202]{Tohuku}, and for practical purposes one can take this as the definition.  Writing 
\[ \pi : V \lra V/G \, =_\text{def}\, W,\] one also has an action of $G$ on $\pi_* F$. The $G$-equivariant direct image of $F$ can be interpreted as
\[   \pi^G_* F \ = \ \big( \pi_*(F)\big)^G, \]
and  one can show that
\[  H^j_G\big(V, F\big) \ = \ H^j\big(W, \pi_*^G F\big).\]  } 
One can think of this as a precision and strengthening of Proposition \ref{Lemma.Surjectivity.Sufficient}. Following the line of attack of Section 1, the plan is to study these groups by modding out by the symmetric group and pushing down to the symmetric product.

To this end, denote by $\Sym^{p+1}(X)$ the $(p+1)^{\text{st}}$ symmetric product of $X$, which we view as parametrizing zero-cycles of degree $p+1$, and write  \[ \pi : X^{p+1} \lra \Sym^{p+1}(X)\] for the quotient map. The equivariant pushforward of the line bundles in \eqref{Equivar.Bundles} determine respectively a line bundle and torsion-free sheaf of rank one \[   \SSS_{p+1}(L) \ =_{\text{def}}  \ \pi_*^{S_{p+1}} \big( L^{\boxtimes p+1, \text{sym}} \big)   \ \  , \ \  \NNN_{p+1}(L) \ =_\text{def} \pi_*^{S_{p+1}} \big( L^{\boxtimes p+1, \text{alt}} \big) \]
on $\Sym^{p+1}(X)$. One has 
\begin{align*}  \HH{0}{\Sym^{p+1}(X)}{\SSS_{p+1}(L) } \ &= \ \Sym^{p+1}H^0(X,L)\\ \HH{0}{\Sym^{p+1}(X)}{\NNN_{p+1}(L) } \ &= \ \Lambda^{p+1}H^0(X,L),\end{align*}
and for any line bundle $A$ on $X$:
\begin{equation} \label{Twisting} \SSS(L\otimes A) \, = \, \SSS(L) \otimes \SSS(A) \ \ , \ \ \NNN(L\otimes A) \, = \, \NNN(L) \otimes \SSS(A). \end{equation}
Moreover $\SSS(A)$ is ample if $A$ is. 

With these preliminaries out of the way, we now give the:
\begin{proof}[Proof of Theorem \ref{Non-Vanishing.Prop}] Note to begin with that the symmetric group $S_{p+1}$ acts on each of the spaces appearing in the diagram  \eqref{Basic.Diagram}.  Taking the quotients yields the diagram:
\begin{equation}   \label{Symmetric.Diagram}
\xymatrix{
\ol{Z} = X \times \Sym^p(X)   \ar[dr]_{\ol{\sigma}} &\subseteq  &    X \times \Sym^{p+1}(X) \ar[dl]^{p_2}   \\ & \Sym^{p+1}(X)
} 
\end{equation}
where $\ol{\sigma}$ is the addition map,  \[ p_1 : X \times \Sym^{p+1}(X)  \lra  X \ \ , \ \ p_2 : X \times \Sym^{p+1}(X) \lra \Sym^{p+1}(X)\]
are  the projections, and the inclusion on the top line is given by
$(x, w) \mapsto (x, x+w)$.  One has
\[  
(1 \times \pi)_*^{S_{p+1}} \big( \pro_1^* B \otimes \pro_2^* L^{\boxtimes p+1, \text{alt}} \big) \ = \ p_1^* \, B \otimes p_2^*\,\NNN(L). 
\] 

Now define 
\[  \GGG(B; L) \ = \ (1 \times \pi)_*^{S_{p+1}} \big( \pro_1^* B \otimes \pro_2^* L^{\boxtimes p+1, \text{alt}} \otimes \OO_Z\big).\]
Pushing forward the restriction to $Z$ gives rise to  a natural surjective mapping
\begin{equation} \label{push.forward.equation.2}  \eps(B;L)\,  : \, p_1^* \, B \otimes p_2^*\,\NNN(L) \lra \GGG(B; L) \end{equation} of sheaves on $X \times \Sym^{p+1}(X)$.
Thanks to \cite[\S 5.2]{Tohuku} the groups appearing in \eqref{equivariant.restriction.equation} are given by the global sections of the sheaves in \eqref{push.forward.equation.2}, and hence under the vanishing hypothesis \eqref{vanishing.hypothesis}, $K_{p,1}(X, B; L)$
is computed as the cokernel \[K_{p,1}(X, B; L) \ = \ 
 \coker \Big( H^0\big(\eps(B;L)\big ) \Big)  \] on global sections determined by $\eps(B;L)$. Hence we are reduced to showing that under the hypothesis of the Proposition, $\eps(B;L_d)$ cannot be surjective on global sections when $d \gg 0$. 

The next step is to form and study the push-forward of \eqref{push.forward.equation.2} to $\Sym^{p+1}(X)$. 
To begin with, define
$\FFF(B;L) = p_{2,*} \GGG(B;L)$ and $\delta(B;L) = p_{2,*} \eps(B;L)$. This gives rise to a morphism 
\[
\delta(B;L) :  H^0(X, B) \otimes_\CC \NNN(L) \lra \FF(B;L) 
\]
of sheaves on $\Sym^{p+1}(X)$ with the property that
\[  K_{p,1}(X,B;L) \ = \ \coker \Big( H^0\big(\delta(B,L)\big)\Big). \]

We wish to study  the geometry of this mapping assuming that $B$ does not impose independent conditions on all reduced cycles. We assert that there is a natural homomorphism
\begin{equation}\label{Def.of.t}  t: \big(\,  p_1^* \, B \otimes p_2^* \, \NNN(L) \otimes \OO_{\ol{Z}} \, \big) \lra \GGG(B; L)\end{equation}
which is an isomorphism on the smooth locus of $\ol{Z}$. Grant this for the time being. By the projection formula one has
\[
p_{2, *}\Big( p_1^*B \otimes p_2^*\, \NNN(L) \otimes \OO_{\ol Z} \Big) \ = \ \ol{\sigma}_* \big( p_1^* B\big) \otimes \NNN(L),
\]
and then taking direct images in  \eqref{push.forward.equation.2}  and \eqref{Def.of.t}, one arrives at a diagram
\begin{equation} \label{Pushdown.Diagram}
\begin{gathered}\xymatrix{
H^0(X, B) \otimes_\CC \NNN(L) \ar[r]^{e\otimes 1} \ar[dr]_{\delta(B;L)} & \ol{\sigma}_*\big( p_1^* B\big) \otimes \NNN(L) \ar[d]^s\\ & \FFF(B;L),
} 
\end{gathered}
\end{equation}
where  $s$ is an isomorphism over the smooth locus of $\Sym^{p+1}(X)$.

Now fix a reduced zero-cycle $w \in \Sym^{p+1}(X)$ that fails to impose independent conditions on $H^0(X, B)$, and let $\xi \subseteq X$ be the corresponding subscheme of length $p+1$. We can identify the  fibre  of the  morphism 
$e : H^0(X, B) \otimes \OO_{\Sym^{p+1}(X)} \lra \ol{\sigma}\big( p_1^* B\big)$ appearing in \eqref{Pushdown.Diagram} at $w$ with the evaluation $H^0(X, B) \lra H^0(X, B \otimes \OO_\xi)$. Writing $\mathcal{K} = \coker ( e)$, so that
\[  \coker ( e \otimes 1) \ = \ \mathcal{K} \otimes \NNN(L), \] it follows that
$ w \ \in \ \textnormal{supp}\, \big(\mathcal{K} \otimes \NNN(L)\big)$.  But thanks to \eqref{Twisting} and Serre vanishing, $\ol{\sigma}_*\big( p_1^* B\big) \otimes \NNN(L_d)$ is globally generated and 
\[  H^0\big( \ol{\sigma}( p_1^* B) \otimes \NNN(L_d)\big)  \lra  H^0\big(\mathcal{K} \otimes \NNN(L_d)\big)\]
is surjective when $d \gg 0$. On the other hand, $s$ is an isomorphism in a neighborhood of $w$ since $w$ is reduced, and it then follows that the map
\[  H^0\big(\mathcal{K} \otimes \NNN(L_d)\big)\lra H^0\big(\coker (\delta(B;L_d)\big)\]
is non-zero when $d \gg 0$. In other words, we have a  commutative diagram
\[
\xymatrix{
H^0(B) \otimes H^0(\NNN(L_d))\ar[r] \ar[dr]_{\delta(B; L_d)}  &H^0\big( \ol{\sigma}( p_1^* B) \otimes \NNN(L_d)\big)  \ar@{>>}[r] \ar[d] \ar[dr]&H^0\big(\mathcal{K} \otimes \NNN(L_d)\big)\ar[d]^{\ne 0}\\
& H^0\big( \FFF(B; L_d)\big) \ar[r] &H^0\big( \coker( \delta(B; L_d)\big),
}
\]
with exact top row, in which the right-hand diagonal mapping, and hence also the bottom homomorphism, are  non-zero. Therefore $\delta(B;L_d)$ cannot be surjective on global sections when $d \gg 0$, as required.

It remains to construct the homomorphism $t$ appearing in \eqref{Def.of.t}. To this end, let
\[
\widehat{Z} \ = \ \ol{Z} \, \times_{X \times \Sym^{p+1}(X)} \big( X \times  X^{p+1}\big).\]
 The projection formula gives an isomorphism
\[
(1 \times \pi)_* \Big( \left( \pro_1^*B \otimes \pro_2^*L^{\boxtimes p+1, \text{alt}} \right) \otimes \OO_{\widehat{Z}}\Big) \ = \ \Big( (1 \times \pi)_* \big( \pro_1^* B \otimes \pro_2^* L^{\boxtimes{p+1, \text{alt}}}\big)\Big) \otimes \OO_{\ol{Z}}, 
\] which, upon taking $S_{p+1}$ invariants, yields
\[
(1 \times \pi)_*^{S_{p+1}} \big( \pro_1^* B \otimes \pro_2^* L^{\boxtimes p+1, \text{alt}} \otimes \OO_{\widehat{Z}}\big) 
 \ = \ \big( p_1^* B \otimes p_2^* \NNN(L) \big) \otimes \OO_{\ol{Z}}. 
\]
On the other hand, as a set $\widehat{Z}$ consists of those points 
\[   \big( x_0, (x_1, \ldots, x_{p+1}) \big) \in X \times X^{p+1} \]
having the property that $x_0$ appears in the cycle $x_1 + \ldots + x_{p+1}$. In other words, $\widehat{Z}$ and $Z$ coincide set-theoretically. Since $Z$ is reduced this implies that $Z = \widehat{Z}_\text{red}$, and in particular $Z$ is a subscheme of $\widehat {Z}$. Thus there is a natural surjective map
\[
\big( \pro_1^* B \otimes \pro_2^* L^{\boxtimes p+1, \text{alt}} \big) \otimes \OO_{\widehat{Z}} \lra \big( \pro_1^* B \otimes \pro_2^* L^{\boxtimes p+1, \text{alt}} \big)\otimes \OO_Z
\]
which is an isomorphism over the smooth locus of $Z$, and taking direct images gives  \eqref{Def.of.t}.
\end{proof}

\begin{remark}
It would be very interesting to give a necessary and sufficient condition for   the non-vanishing of $K_{p,1}(X, B;L_d)$ for $d \gg 0$. Keeping the notation of the previous proof, the issue is to determine when $\delta(B;L)$ has a non-zero cokernel. It is conceivable that  the failure of $B$ to be $p$-jet very ample suffices, but the question seems somewhat difficult to analyze. Already the case $\dim X = 2$ would be interesting.  \end{remark}

\begin{remark} Recall that a line bundle $B$ on a smooth variety $X$  is said to be $p$-very ample if every finite subscheme of length $p+1$ imposes independent conditions on $H^0(X,B)$. When $\dim X = 1$ this is the same as jet-amplitude, but when $\dim X \ge 2$ it is a strictly weaker condition in general. A quick way to see this is to recall that if $A$ is  an ample line bundle on a smooth surface $X$, then $B_p = K_X + (p+3)A$ is always $p$-very ample thanks to a theorem of Beltrametti and Sommese [BS]. On the other hand, the $p$-jet amplitude of $B_p$ for $p \gg 0$ would imply that  the Seshadri constant $\eps(A;x)$ is very close to $1$ for every point $x \in X$ (see \cite[Proposition 5.10]{LLS}). Hence any line bundle $A$ for which there exist points with small Seshadri constant gives rise to examples of the required sort. 

\end{remark}
 %
 %
 %
 %

 \end{document}